\documentclass[12pt]{article}

\usepackage[utf8]{inputenc}

\usepackage[a4paper,nomarginpar]{geometry}
\usepackage[english]{babel}

\usepackage{graphicx}
\usepackage{amsmath,amsfonts,amssymb,amsthm}
\usepackage{amscd}
\usepackage{xcolor}

\theoremstyle{definition}
\newtheorem{theorem}{Theorem}

\newtheorem{corollary}[theorem]{Corollary}

\theoremstyle{definition}

\newtheorem{example}[theorem]{Example}

\theoremstyle{remark}
\newtheorem{remark}[theorem]{Remark}

\newcommand{\N}{\mathbb{N}} 
\newcommand{\Z}{\mathbb{Z}} 
\newcommand{\R}{\mathbb{R}} 

\newcommand{\abs}[1]{\left\lvert#1\right\rvert}

\newcommand{\norm}[1]{\lVert#1\rVert}

\newcommand{\spa}{\operatorname{span}}

\newcommand{\eps}{\varepsilon}
\newcommand{\Tt}{\ensuremath{(T_t)_{t\geq 0}}}

\begin{document}

\title{Strong mixing measures for $C_0$-semigroups
 \thanks{This work is supported in part by MICINN and FEDER, Project
MTM2010-14909. The first author was also supported by a grant from the FPU Program of MEC.}}


\author{M. Murillo-Arcila\footnote{IUMPA, Universitat Polit\`{e}cnica de Val\`{e}ncia, Edifici 8G, Cam\'{\i} Vera S/N, 46022 Val\`{e}ncia, Spain.  e-mail: {mamuar1@posgrado.upv.es}} and  A. Peris\footnote{IUMPA, Universitat Polit\`{e}cnica de Val\`{e}ncia, Departament de Matem\`{a}tica Aplicada, Edifici 7A, 46022 Val\`{e}ncia, Spain. e-mail: aperis@mat.upv.es}}

\date{ }

\maketitle

\begin{abstract}
Our purpose is to obtain a very effective and general method to prove that certain $C_0$-semigroups admit invariant strongly mixing measures.
More precisely, we show that the Frequent Hypercyclicity  Criterion  for $C_0$-semigroups ensures the existence of invariant mixing measures with full support.
We will several examples, that range from birth-and-death models to the Black-Scholes equation, which illustrate these results.

\end{abstract}

\section{Introduction}

The interest in the dynamics of $C_0$-semigroups of operators comes from the analysis of the asymptotic behaviour of solutions to certain linear partial differential equations and to infinite systems of linear differential equations. Especially, the chaotic behaviour (in the topological and in the measure-theoretic sense) of $C_0$-semigroups has experimented a great development in recent years (see, e.g., \cite{albanese_et_alt,banasiak_moszynski2011,barrachina_peris,bayart_SF_11,bermudez_et_alt,conejero_mangino,conejero_peris_dcds_09,ji-weber,kalmes,Rdn12invariantmeasures}).

We recall that $(T_t)_{t\geq 0}$ is a  \emph{$C_0$-semigroup} if  $T_0=I$, $T_{t+s}=T_t\circ T_s$ and $\lim_{s\rightarrow t}T_sx=T_tx$ for all $x\in X$ and $t\geq 0$.
If $X$ is a separable infinite-dimensional Banach space, a \emph{$C_0$-semigroup} $(T_t)_{t\geq 0}$ of linear and continuous operators on $X$ is said to be \emph{hypercyclic} if there exists $x\in X$ such that the set $\{T_tx:t\geq 0\}$ is dense in $X$. An element $x\in X$ is  a \emph{periodic point} for the semigroup if there exists $t>0$ such that $T_tx=x$. A semigroup $(T_t)_{t\geq 0}$ is called \emph{chaotic} if it is hypercyclic and the set of periodic points is dense in $X$.

There are analogous properties related to $C_0$-semigroups defined on a probability space $(X,\mathfrak{B},\mu)$, where $X$ is a Banach space  and $\mathfrak{B}$ denotes the algebra of Borel subsets of $X$. We will say that a Borel probability measure $\mu$ has \emph{full support} if for any non-empty open set $U\subset X$ we have $\mu(U)>0$. A $C_0$-semigroup is \emph{mixing} if
$$
\lim_{t\rightarrow\infty}\mu(A\cap T_{t}^{-1}(B))=\mu(A)\mu(B)\qquad (A,B\in\mathfrak{B}).
$$
Mixing implies \emph{ergodicity}, i.e., the invariance $T^{-1}(A)=A$ for certain $A \in \mathfrak{B}$ implies, either $\mu (A)=0$, or $\mu (A)=1$.

The concept of frequent  hypercyclicity was introduced by Bayart and Grivaux \cite{BaGr06} inspired by Birkhoff's ergodic theorem. The first ones that used ergodic theory for the dynamics of linear operators were  Rudnicki \cite{Rdn93} and Flytzanis \cite{Fly95}. The notion of frequent hypercyclicity was extended to $C_0$-semigroups in \cite{BaGri07}. We recall that the lower density of a measurable  set $M\subset \mathbb{R_+}$ is defined by
$$
\underline{Dens}(M):=\liminf_{N\rightarrow \infty}\frac{\mu(M\bigcap[0,N])}{N},
$$
where $\mu$ is the Lebesgue measure on $\mathbb{R_+}$. A $C_0$-semigroup $(T_t)_{t\geq 0}$ is said to be  \emph{frequent hypercyclic} if there exists $x\in X$  such that $\underline{Dens}(\{t\in\mathbb{R_+} \ ; \ T_tx\in U\})>0$ for any non-empty open set  $U\subset X$.

The first version of a Frequent Hypercyclicity Criterion for operators was obtained by Bayart and Grivaux \cite{BaGr06}. Later, Bonilla and Grosse-Erdmann \cite{bonilla_grosse-erdmann2007frequently}  gave a more general formulation for operators on separable $F$-spaces. Another (probabilistic) version of it was provided by Grivaux \cite{Gri06}.

In \cite{mangino_peris2011frequently}, Mangino and Peris obtained a continuous version  of the frequent hypercyclicity criterion based on the Pettis integral. This criterion can be verified in certain cases in terms of the infinitesimal generator of the semigroup. They also gave applications for $C_0$-semigroups generated by Ornstein-Uhlenbeck operators, and  for translation semigroups on weighted spaces of $p$-integrable (or continuous) functions. Their main result  was the following sufficient condition for frequent hypercyclicity.

\begin{theorem}[\cite{mangino_peris2011frequently}]\label{fhcs}
Let $(T_t)_t$ be a $C_0$-semigroup on a separable Banach space $X$. If there exist, $X_0\subset X$ dense in $X$, and maps $S_t:X_0\rightarrow X_0$, $t>0$, such that
\begin{itemize}
\item[(i)] $T_tS_tx=x,T_tS_rx=S_{r-t}x, t>0, r>t>0$,
\item[(ii)] $t\rightarrow T_tx$ is Pettis integrable in $[0,\infty)$ for all $x\in X_0$,
\item[(iii)] $t\rightarrow S_tx$ is Pettis integrable in $[0,\infty)$ for all $x\in X_0$.
\end{itemize}
then  $(T_t)_{t\geq 0}$  is frequently hypercyclic.
\end{theorem}

Our purpose is to show that this criterion suffices for the existence of invariant Borel probability measures on $X$ that are strongly mixing and have full support. We also refer the reader to the recent paper of Bayart and Matheron \cite{bayart_matheron0000mixing}, that offers very general conditions which ensure the existence of mixing measures. 

This is a continuous version of a result that we obtained for single operators \cite{murillo_peris_jmaa_13}. More precisely, under the hypothesis of Bonilla and Grosse-Erdmann \cite{bonilla_grosse-erdmann2007frequently},  the authors derived a stronger result by showing that a   $T$-invariant strongly mixing measure with full support can be obtained.
 \begin{theorem}[\cite{murillo_peris_jmaa_13}]\label{fhc}
Let $T$ be an operator on a separable F-space $X$. If there are, a dense subset $X_0$ of $X$, and a sequence of maps $S_n:X_0\rightarrow X$ such that, for each $x\in X_0$,
\begin{itemize}
\item[(i)]$\sum_{n=0}^\infty T^nx$ converges unconditionally
\item[(ii)]$\sum_{n=0}^\infty S_nx$ converges unconditionally, and
\item[(iii)]$T^nS_nx=x$ and $T^mS_nx=S_{n-m}x$ if $n>m$.
\end{itemize}
then there is a $T$-invariant strongly mixing Borel probability measure $\mu$ on $X$ with full support.
\end{theorem}

 In contrast with the chaotic behaviour in the topological sense, which is trivial to pass from the discrete to the continuous case, while difficult to go in the other direction (see, e.g., \cite{conejero_muller_peris} for hypercyclicity and frequent hypercyclicity), the measure-theoretic properties are not trivially passed from the discrete to the continuous case, especially because of the requirement of $T_t$-invariance for every $t>0$. This is why we need to construct explicitly the mixing measures for $C_0$-semigroups, and they cannot be obtained from the main result in  \cite{murillo_peris_jmaa_13}.


Our notation is standard, and refer to the recent books \cite{bayart_matheron2009dynamics} and \cite{grosse-erdmann_peris2011linear} for the basic theory
on chaotic linear dynamics.

We also recall the main definitions and results about Pettis integrability that will be needed in the paper.
The proofs of all these results  can be found in \cite{DU} for the case of
finite measure space, but they easily extend to $\sigma$-finite measure spaces.
Let $X$ be a Banach space and $(\Omega,\mu)$  a $\sigma$-finite measure space. A function
$f:\Omega\rightarrow X$ is said to be \emph{weakly $\mu$-measurable} if the scalar function $\varphi\circ f$
is $\mu$-measurable for every $\varphi\in X'$, where  $X'$  denotes the topological dual of $X$;
$f$ is said to be \emph{$\mu$-measurable} if there exists a sequence $(f_n)_n$ of simple functions
 such that $\lim_{n\to\infty}\abs{f_n-f}=0$ $\mu$-a.e.

Dunford's lemma says that, if $f$ is weakly $\mu$-measurable and $\varphi\circ f\in L_1(\Omega,\mu)$
for every $\varphi\in X'$, then for every measurable $E\subseteq \Omega$ there exists $x_E\in X''$ such that
\[
x_E(\varphi)=\int_E \varphi\circ f d\mu,
\]
for every $\varphi\in X'$.
When $f:\Omega \rightarrow X$ is weakly $\mu$-measurable and
 $\varphi\circ f\in L_1(\Omega,\mu)$ for every $\varphi\in X'$, then $f$ is called \emph{Dunford
integrable}. The Dunford integral of $f$ over a measurable $E\subseteq \Omega$ is defined by
the element $x_E\in X''$ such that
\[
x_E(\varphi)=\int_E \varphi\circ f d\mu,
\]
for every $\varphi\in X'$.

In the case that $x_E\in X$ for every measurable $E$, then $f$ is called \emph{Pettis integrable} and
$x_E$ is called the Pettis integral of $f$ over $E$ and will be denoted by
$(P)-\int_E f d\mu$.
  Clearly the Dunford and Pettis integrals coincide if $X$ is a reflexive space.
Moreover, if $\norm{f}$ is integrable on $\Omega$ (i.e. $f$ is \emph{Bochner integrable} on $\Omega$),
then $f$ is Pettis integrable on $\Omega$. A basic result of Pettis says that, if $f$ is Pettis integrable, then for every sequence
$(E_n)_n$ of disjoint measurable sets in $\Omega$
\[
\int_{\bigcup_{n\in\N} E_n}fd\mu =\sum_{n\in\N} \int_{E_n} fd\mu,
\]
where the series converges unconditionally. As a consequence,
if $f:[0,+\infty[\rightarrow X$ is   Pettis integrable on $[0,+\infty[$, then
for every $\varepsilon>0$ there exists $N>0$ such that for every compact set $K\subset [N,+\infty[$
\[
\norm{\int_K f(t)dt}<\varepsilon.
\]

\section{Invariant measures and the frequent hypercyclicity criterion}

We are now ready to present our main result. The idea behind the proof is to construct, given a $C_0$-semigroup $(T_t)_{t>0}$ on a separable Banach space $X$ satisfying the hypothesis of Theorem~\ref{fhcs},
\begin{enumerate}
\item a ``model'' probability space $(\textit{Z},\overline{\mu})$ and
\item a Borel measurable map $\Phi:\textit{Z}\rightarrow X$ with dense range,
\end{enumerate}
where
\begin{itemize}
\item $\textit{Z}\subset C(\mathbb{R})$ is  a $(R_t)_{t\in \R}$-invariant subset of the space $C(\R )$ of continuous functions on the real line,
\item $(R_t)_{t\in \R}$ is the translation  group,
\item $\overline{\mu}$ is a $(R_t)_{t\in \R}$-invariant strong mixing measure with full support, and
\item $\Phi R_t=T_t\Phi$ for all $t\geq 0$.
\end{itemize}
As a consequence, the Borel probability measure $\mu$ on $X$ defined by $\mu(A)=\overline{\mu}(\Phi^{-1}(A))$, $A\in\mathfrak{B}(X)$, is $(T_t)_{t>0}$-invariant, strongly mixing, and has full support.

\begin{theorem}\label{main}
Let $(T_t)_t$ be a $C_0$-semigroup on a separable Banach space $X$. If there exist, $X_0\subset X$ dense in $X$, and maps $S_t:X_0\rightarrow X_0$, $t>0$, such that :
\begin{itemize}
\item[(i)] $T_tS_tx=x,T_tS_rx=S_{r-t}x, t>0, r>t>0$,
\item[(ii)] $t\rightarrow T_tx$ is Pettis integrable in $[0,\infty)$ for all $x\in X_0$,
\item[(iii)] $t\rightarrow S_tx$ is Pettis integrable in $[0,\infty)$ for all $x\in X_0$.
\end{itemize}
then there is a $(T_t)_t$-invariant strongly mixing Borel probability measure $\mu$ on $X$ with full support.
\end{theorem}

\begin{proof}
\vspace{3mm}
We suppose $X_0=\{x_n; n\in\mathbb{N}\}$ with $x_1=0$. Let $U_n=B(0,\frac{1}{2^n})$, the open ball of radius $1/2^n$ centered at $0$. By conditions (ii) and (iii)  we can obtain an increasing sequence $\{N_n\}_n\in\mathbb{N}$  with $N_{n+2}-N_{n+1}>N_{n+1}-N_n$ for all $n\in\mathbb{N}$ such that,  for any sequence $(C_k)_k$   of mutually disjoint compact sets with $C_k\subset [k/2,+\infty [$, $k\in\N$, we have that
\[
\sum_{k\geq N_n}\int_{C_k}T_tx_{m_k} dt \in U_{n+1}\quad\mbox{and}\quad\sum_{k\geq N_n}\int_{C_k}S_tx_{m_k} dt \in U_{n+1}
\]
\begin{equation}\label{unconditional}
\mbox{if }  m_k\leq 2l, \mbox{ for } N_l\leq k< N_{l+1}, \ \ l\geq n, \ \ n\in\N .
\end{equation}

\noindent \textbf{1.-The model probability space $(Z,\overline{\mu})$.}

\vspace{3mm}
  First of all, we define  the following set $\textit{A}\subset C(\R)$ of continuous functions: $f\in\textit{A}$ if there exist a sequence $(s_i)_{i\in\mathbb{Z}}$ of real numbers such that $\dots s_{-4}<s_{-2}<0\leq s_0<s_2<s_4<\dots $, $\frac{1}{2}<s_{2i+2}-s_{2i}<\frac{3}{2}$, $s_{2i+1}=(s_{2i}+s_{2i+2})/2$, $i\in\Z$, and a sequence of natural numbers $(n_i)_{i\in\mathbb{Z}}$ such that $f(s_{2i})=n_i$, $f(s_{2i+1})=0$  and $f''|_{]s_i,s_{i+1}[}\equiv 0$   for all $i\in\mathbb{Z}$.

\begin{figure}
\begin{center}\scalebox{.65}{\includegraphics{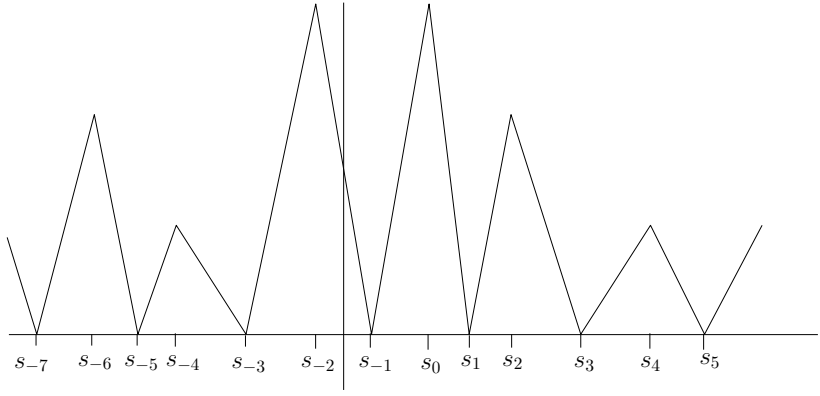}}\end{center}\caption{Graph of a typical function $f\in A$}\label{figure1}
\end{figure}

 Let us define, for each $\alpha=((s_{2j})_{j=-n}^m, (n_j)_{j=-n}^m,\varepsilon)\in \R^{n+m+1}\times \N^{n+m+1}\times ]0,1/4[$ with $s_{-2n}<\dots <s_{-2}<0\leq s_0<s_2< \dots<s_{2m} $, $\frac{1}{2}<s_{2j+2}-s_{2j}<\frac{3}{2}$,  $j=-n, \dots ,m-1$,  the set
 $$
 A_{\alpha}=\{f\in\textit{A} \ ; \ \exists t_{2j}\in ]s_{2j}-\varepsilon,s_{2j}+\varepsilon[ \mbox{ with } f(t_{2j})=n_j, \ f(t_{2j+1})=0
 $$
 $$
 \mbox{for } t_{2j+1}:=\frac{t_{2j}+t_{2j+2}}{2}, \ j=-n,\dots ,m-1, \ f''|_{]t_i,t_{i+1}[}\equiv 0, \ i=-2n, \dots ,2m-1 \} .
 $$
 They form a base of open sets in $\textit{A}$   as a topological subspace of $C(\mathbb{R})$ endowed with the compact-open topology.

We will introduce a measure in $\textit{A}$, and it suffices  to define it in the base of open sets considered above for sufficiently small $\eps$.
Actually, if $|s_{2j+2}-s_{2j}-1|<\frac{1}{2}-2\eps$, $j=-n,\dots , m-1$, we set
\begin{equation}\label{measurebase}
\overline{\mu}(A_{((s_{2j})_{j=-n}^m,(n_j)_{j=-n}^m,\varepsilon )})=\prod_{j=-n}^m2\varepsilon p(\{n_j\}),
\end{equation}
where $p$ is the probability measure defined in $\mathbb{N}$ such that $p(\{n_j\})=p_j$, with $0<p_j<1$, $\sum_{j=1}^\infty p_j=1$ and, if
\begin{equation}\label{p_j}
\beta_j:=(\sum_{i=1}^jp_i)^{N_{j+1}-N_j}, j\in\mathbb{N},\quad\mbox{then}\quad\prod_{j=1}^\infty\beta_j>0.
\end{equation}
As an easy application of Fubini's Theorem, and because of the selection of the $s_i$, it is clear that $\overline{\mu}$ is a full support Borel probability measure on $A$. Moreover,  $\textit{A}$ is
$R_t$-invariant for any $t\in\R$,  where $(R_t)_{t\in\mathbb{R}}$ is the translation $C_0$-group, since given $f_{(s_j,n_j)_j}\in\textit{A}$ we have that  $R_t(f_{(s_{j},n_j)_j})=f_{(t+s_{j+2k},n_{j+k})_j}\in\textit{A}$, where
\begin{equation}\label{despl}
k:=\min\{j\in\mathbb{Z} \ ; \ t+s_{2j}\geq 0\} .
\end{equation}
The definition of $\overline{\mu}$ given in \eqref{measurebase} yields that $\overline{\mu}$ is $(R_t)_{t\in\mathbb{R}}$-invariant.

 We also note that $\overline{\mu}$ is strongly mixing with respect to the translation $C_0$-group $(R_t)_{t\in\mathbb{R}}$:
 Let  $A_{\alpha}$ and $A_{\alpha'}$ be two elements from the above base of open sets in $\textit{A}$, where $\alpha=((s_{2j})_{j=-n}^m, (n_j)_{j=-n}^m,\varepsilon)$ and $\alpha=((s'_{2j})_{j=-n'}^{m'}, (n'_j)_{j=-n'}^{m'},\varepsilon')$. If $t$ is large enough then $[s_{-n}-\varepsilon,s_m+\varepsilon]\bigcap [t+s'_{-n'}-\varepsilon',t+s'_{m'}+\varepsilon']=\emptyset$  and
 $$
 \mu(A_{\alpha}\cap R_{t}(A_{\alpha'}))=\mu(A_{\alpha})\mu(A_{\alpha'}).
 $$
 Let us consider the compact subset of $A$ given by
 $$
 H=\{f_{(s_k,n_k)_k}\in A \ ; \ n_k=f(s_{2k})\in\{1,\ldots,m\} \mbox{ if } N_m\leq |k|< N_{m+1},
 $$
 $$
  m\in \N, \ f(s_{2k})=1  \mbox{ for }  |k|< N_1\}.
 $$
Let  $Z=\bigcup_{t\in\mathbb{R}}R_t(H)=\bigcup_{j\in\Z} R_j(H)$, therefore a Borel subset of $A$. We easily get
$$
\overline{\mu}(Z)\geq \overline{\mu}(H)
 =( p_1)^{2N_1-1}(\prod_{l=1}^\infty \beta_l)^2>0.
$$
Since $Z$ is $R_t$-invariant and it has positive measure, then $\overline{\mu}(Z)=1$

\vspace{3mm}

\noindent \textbf{2.-The map $\Phi$.}

\vspace{3mm}

Given $t_0\in\mathbb{R}$ we define the map  $\Phi:R_{t_0}(H)\rightarrow X$  by
\begin{equation}\label{themap}
\Phi(f_{(s_j,n_j)_j})=\sum_{j\leq -2}\int_{s_{2j}}^{s_{2j+2}}S_{-t}x_{n_j}+ \int_{s_{-2}}^0 S_{-t}x_{n_{-1}}+\int_{0}^{s_0}T_tx_{n_{-1}}+\sum_{j\geq 0}\int_{s_{2j}}^{s_{2j+2}}T_tx_{n_j}
\end{equation}
 $\Phi$ is well defined since, given $f_{(s_j,n_j)_j}\in R_{t_0}(H)$, and for $l\geq |t_0|$, we have that $n_k\leq 2l$ if $N_l< |k|\leq N_{l+1}$, which shows the convergence of the series in \eqref{themap} by \eqref{unconditional}.
  Let us see that $T_a\circ\Phi=\Phi\circ R_a$ for any $a>0$. We will distinguish two cases:
\begin{itemize}
\item[Case 1] $0<a<-s_{-2}$:
$$
T_a\circ\Phi(f_{(s_j,n_j)_j})=\sum_{j\leq  -2}\int_{a+s_{2j}}^{a+s_{2j+2}}S_{-t}x_{n_j}+\int_{a+s_{-2}}^0 S_{-t}x_{n_{-1}}+\int_{0}^{a+s_0}T_{t}x_{n_{-1}}
$$
$$
+\sum_{j\geq 0}\int_{a+s_{2j}}^{a+s_{2j+2}}T_{t}x_{n_j}=\Phi(f_{(a+s_{j},n_{j})_j})=\Phi\circ R_a(f_{(s_j,n_j)_j})
$$
since, in this case, $0=\min\{j\in\mathbb{Z} \ ; \ a+s_{2j}\geq 0\}$

\item[Case 2] $s_{2k}<-a\leq s_{2k+2}$, for some $k\in\mathbb{Z}^{-}$, $k\leq -2$:

$$
T_a\circ\Phi(f_{(s_j,n_j)_j})=\sum_{j<k } \int_{a+s_{2j}}^{a+s_{2j+2}}S_{-t}x_{n_j}+\int_{a+s_{2k}}^0 S_{-t}x_{n_{k}}+\int_{0}^{a+s_{2k+2}}T_{t}x_{n_{k}}
$$
$$
+\sum_{j>k }\int_{a+s_{2j}}^{a+s_{2j+2}}T_{t}x_{n_j}=\Phi(f_{(a+s_{j+2k+2},n_{j+k+1})_j})=\Phi\circ R_a(f_{(s_j,n_j)_j})
$$
since, in this case, $k+1=\min\{j\in\mathbb{Z} \ ; \ a+s_{2j}\geq 0\}$.
\end{itemize}

Also, $\Phi$ is continuous almost everywhere on $R_{t_0}(H)$ for any $t_0\in\R$. Indeed, let $(f_{(s_j^k,n_j^k)_j})_k$ be a sequence in $R_{t_0}(H)$ that converges to $f_{(s_j,n_j)_j}\in R_{t_0}(H)$ with  $s_0>0$.
Then, for any compact set $C\subset \R$, we have that
$$
\lim_{k \rightarrow\infty}\sup_{x\in C}d(f_{(s_j^k,n_j^k)_j}(x),f_{(s_j,n_j)}(x))=0.
$$
In particular, for any $N\in\mathbb{N}$ and $\varepsilon>0$, there exists $n_0\in\mathbb{N}$ such that,
\begin{equation}\label{a}
\mbox{if } |j|\leq N  \mbox{ and }  k\geq n_0,  \mbox{ then }  n_j^k=n_j \mbox{ and }  |s_{2j}^k-s_{2j}|<\varepsilon.
\end{equation}

Fix $n>|t_0|$ and $N=N_n$. Let $\eps>0$ such that $\norm{\int_I S_{-t}x_{n_j}dt}+\norm{\int_J T_{t}x_{n_j}dt}<(3(N+1)2^{n+1})^{-1}$
whenever $I\subset ]-\infty,0]$ and $J\subset [0,+\infty [$ are intervals
of length less than $\eps$ and $|j|\leq N$. By \eqref{a} and \eqref{unconditional},
$$
\|\Phi(f_{(s_j^k,n_j^k)_j}^k)-\Phi(f_{(s_j,n_j)})\|\leq
\norm{\sum_{j<-N_n}\int_{s_{2j}^k}^{s_{2j+2}^k}S_{-t}x_{n_j}^k} +
\norm{\sum_{j>N_n}\int_{s_{2j}^k}^{s_{2j+2}^k}T_{t}x_{n_j}^k}
$$
$$
+\norm{\sum_{j<-N_n}\int_{s_{2j}}^{s_{2j+2}}S_{-t}x_{n_j}}+
\norm{\sum_{j>N_n}\int_{s_{2j}}^{s_{2j+2}}T_{t}x_{n_j}}
+\sum_{-N_n\leq j\leq -2}\norm{\int_{\min{(s_{2j}^k,s_{2j})}}^{\max{(s_{2j}^k,s_{2j})}}S_{-t}x_{n_{2j}}}
$$
$$
+ \sum_{-N_n\leq j\leq -2}\norm{\int_{\min{(s_{2j+2}^k,s_{2j+2})}}^{\max{(s_{2j+2}^k,s_{2j+2})}}S_{-t}x_{n_{j}}}+
\norm{\int_{\min{(s_{-2}^k,s_{-2})}}^{\max{(s_{-2}^k,s_{-2})}}S_{-t}x_{n_{-1}}}+
\norm{\int_{\min{(s_{0}^k,s_{0})}}^{\max{(s_{0}^k,s_{0})}}T_{t}x_{n_{-1}}}
$$
$$
 \sum_{0\leq j\leq N_n}\norm{\int_{\min{(s_{2j}^k,s_{2j})}}^{\max{(s_{2j}^k,s_{2j})}}T_{t}x_{n_{j}}}
+ \sum_{0\leq j\leq N_n}\norm{\int_{\min{(s_{2j+2}^k,s_{2j+2})}}^{\max{(s_{2j+2}^k,s_{2j+2})}}T_{t}x_{n_{j}}} < \frac{1}{2^{n-1}}+\frac{1}{2^n}.
$$

 This shows the continuity almost everywhere of $\Phi:R_t(H)\rightarrow X$ for every $t\in\mathbb{R}$.
 The map $\Phi$ is well-defined on $Z$, and $\Phi:Z\rightarrow X$ is measurable (i.e., $\Phi^{-1}(A)\in\mathfrak{B}(Z)$ for every $A\in\mathfrak{B}(X)$).

\vspace{3mm}

\noindent \textbf{3.-The measure $\mu$ on $X$.}

\vspace{3mm}

 $L(t):=\Phi(R_t(H))$ is compact in $X$, $t\in\mathbb{R}$, and $Y:=\bigcup_{t\in\mathbb{R}} L(t)$ is  a   $T_t$-invariant Borel subset of $X$ because $\Phi\circ R_t=T_t\circ\Phi$.

We then define in $X$ the  measure $\mu(B)=\overline{\mu}(\Phi^{-1}(B))$ for all $B\in\mathfrak{B}(X)$. Obviously, $\mu$ is well-defined and it is a $(T_t)_t$-invariant strongly mixing Borel probability  measure.
The proof is completed by showing that $\mu$ has full support. In the proof of \cite[Theorem 2.2]{mangino_peris2011frequently} we showed that, for
$u_k:=\int_0^1 T_tx_kdt$, $k\in\N$, the set $\{ u_k \ ; \ k\in\N\}$ is dense in $X$. Thus, given a non-empty open set $U$ in $X$, we pick $n\in\N$ and $\eps>0$ satisfying
$$
\int_{s_0}^{s_2} T_tx_ndt+ U_n \subset U
$$
for any $s_0 \in [0,\eps ]$, $s_2 \in [1,1+\eps]$.
 Together with \eqref{unconditional}, this implies

 $$
\mu(U)\geq \mu\left(\left\{   \Phi(f_{(s_j,n_j)_j}) \ ; \   f_{(s_j,n_j)_j}\in Z, \  s_0 \in [0,\eps ], \ s_2 \in [1,1+\eps], \right. \right.
$$
$$
\left. \left.  n_0=n, \ n_k=1 \mbox{ if } 0<|k|\leq N_n,   \  n_k\leq 2l, \mbox{ for } N_l< |k|\leq N_{l+1}, \   l\geq n      \right\}\right)
$$
$$
\geq\eps^2p_n( p_1)^{2N_n}\prod_{l=n}^\infty\left(\prod_{N_l<|k|\leq N_{l+1}}\sum_{r=1}^{2l} p_r\right)>\eps^2p_n( p_1)^{2N_n}\prod_{l=n}^\infty\left(\beta_{l}\right)^2>0
$$

\end{proof}

\begin{remark}
There exists an alternative way of defining the measure on the space of continuous functions, by using Brownian motions (for more details see \cite{Rdn85},\cite{Rdn88}). We denote by $\mathfrak{B}=\mathfrak{B}(C([0,\infty)))$, the $\sigma$-algebra of Borel subsets of $C([0,\infty))$. Let $\omega_t$, $t\geq 0$, be a Brownian motion defined on a probability space $(\Omega,\Sigma,P)$. Assume that the sample functions of $\omega_t$ are continuous. Now, if we denote by $F_A$ the $\sigma$-algebra of events generated by the process $\omega_t$ for $t\in A$. Setting $\xi_t=e^t\omega_{e^{-2t}}$ for $t\geq 0$, then $\xi_t$ is a stationary Gaussian process with mean value $E\xi_t=0$ and correlation function $E \xi_t\xi_{t+h}=e^{-|h|}$. Then  the measure on $\mathfrak{B}=\mathfrak{B}(C([0,\infty)))$ induced by $\xi_t$ is strongly mixing with full support.\\
\end{remark}

In \cite[Cor. 2.3]{mangino_peris2011frequently} some conditions, expressed in terms of eigenvector fields for the infinitesimal generator of the $C_0$-semigroup, were obtained to satisfy the hypothesis of Theorem~\ref{main}.
In consequence we also obtain the stronger result of existence of invariant mixing measures under the same conditions.

\begin{corollary}\label{cor:dsw} Let $X$ be a separable complex Banach space and let
$\Tt$ be a $C_0$-semigroup on $X$ with generator $A$.
 Assume that there exists a family $(f_j)_{j\in \Gamma }$ of locally bounded measurable maps
 $f_j:I_j \rightarrow X$ such that $I_j$ is an interval in $\R$,
                    $Af_j(t)=itf_j(t)$ for every $t \in I_j$, $j\in \Gamma$ and
                     $\spa \{f_j (t) \ : \ j\in \Gamma,\ t\in I_j\}$ is dense in $X$.
If either

a) $f_j \in C^2(I_j, X)$, $j\in \Gamma$,\\
or

b) $X$ does not contain $c_0$  and $\langle \varphi,f_j\rangle \in C^{1}(I_j)$, $\varphi\in X'$, $j\in\Gamma$,\\
then there is a $(T_t)_t$-invariant strongly mixing Borel probability measure $\mu$ on $X$ with full support.
\end{corollary}

\section{Applications}

In this section we will present several applications of the previous results to the (chaotic) behaviour of the solution $C_0$-semigroup to certain linear partial differential equations and infinite systems of linear differential equations.

\begin{example}
Let us consider  the following linear perturbation of the one-dimensional
Ornstein-Uhlenbeck operator
\[
{\mathcal{A}_\alpha}u=u''+bxu'+\alpha u,
\]
where $\alpha\in \R$, with domain
\[
  D(\mathcal{A}_\alpha)=\left\{\, u\in L^2(\R)\cap
  W^{2,2}_{{\rm loc}}(\R) \ ; \  \mathcal{A}_\alpha u\in L^2(\R)\,\right\}.
\]
We know that, if $\alpha>b/2>0$, then the semigroup generated by
$\mathcal{A}_\alpha$ in $L^2(\R)$ is chaotic \cite{conejero_mangino} and frequently hypercyclic \cite{mangino_peris2011frequently}. Actually, it was shown that the $C_0$-semigroup satisfies the hypothesis of Corollary~\ref{cor:dsw} \cite{mangino_peris2011frequently}. Therefore, we also obtain that it admits an invariant strong mixing measure with full support.
\end{example}

\begin{example} Rudnicki \cite{Rdn12invariantmeasures}  recently showed the existence of invariant mixing measures for some $C_0$-semigroups generated by a partial differential equation of population dynamics.
More precisely, he reduced the equation to
\[
\frac{\partial u}{\partial t} + x\frac{\partial u}{\partial x} = au(t,x)+bu(t,2x) ,
\]
whose formal solution, given the initial condition $u(0,x)=u_0(x)$, is
\[
u(t,x) :=e^{at} \sum_{n=0}^\infty \frac{(bt)^n}{n!} u_0(2^ne^{-t}x).
\]
He considered the space
$$
X=X_{\alpha,\beta}:=\{u\in C(]0,\infty[) \ ; \  \lim_{x\rightarrow 0}x^\alpha|u(x)|=0, \ \lim_{x\rightarrow \infty}x^\beta|u(x)|=0\}
$$
endowed with the norm $\norm{u}:=\sup_{x\in ]0,\infty [} |u(x)| \rho (x)$, where  $\rho (x)=x^\alpha$ if $x\leq 1$ and $\rho (x)=x^\beta$ if $x>1$.
If $2^ab\log 2<e^{-1}$, $\beta <\log_2b+\log_2 (\log 2)$, and $\alpha >\alpha_0$, where $\alpha_0$ satisfies $(a+\alpha_0)2^{\alpha_0}=b$, then there
exists a Borel probability measure $\mu$ on $X$ which is invariant under the solution $C_0$-semigroup generated by the above equation, is strongly mixing, and has full support \cite[Thm 1]{Rdn12invariantmeasures}. Actually, this fact was shown by reducing the problem to the translation flow $(R_t)_{t\in\R}$ on the space
\[
Y:=\{ g\in C(\R ) \ ; \ \lim_{|x|\to \infty} \frac{g(x)}{x}=0\},
\]
of weighted continuous functions with the norm
\[
\norm{g}_Y=\sup_{x\in \R} \frac{|g(x)|}{1+|x|} .
\]
The corresponding generator is $A=D$, the derivative operator. We can apply directly our Corollary~\ref{cor:dsw} to the map $f:\R\to Y$ given by $[f(t)](x):=e^{itx}$, which is a $C^2$-map, and obtain the same result since $\spa \{f(t) \ ; \ t\in \R\}$ is the set of trigonometric polynomials, which is dense in $Y$.

\end{example}

\begin{example} The chaotic behaviour associated to birth-and-death processes has been widely studied by Banasiak et alt \cite{banasiak_lachowicz2001,banasiak_lachowich_moszynski2006,banasiak_lachowich_moszynski2007,banasiak_moszynski2011}. We will consider three  cases that are shown to admit invariant mixing measures.
\begin{enumerate}
\item
In \cite{banasiak_moszynski2011} Banasiak and Moszynski studied the following ``birth-and-death'' model  with constant coefficients:

\begin{equation}\label{model_5}
    \begin{array}{ccclr}
      \frac{df_1}{dt} & = & (\mathcal{L}f)_1 = & af_1 + df_2, &  \\
       \\
      \frac{df_n}{dt} & = & (\mathcal{L}f)_n = & bf_{n-1} + af_n + df_{n+1}, & \ \ n\geq 2.
    \end{array}
\end{equation}

Among other things, they studied the chaotic behaviour of the solution $C_0$-semigroup.

\begin{theorem}[\cite{banasiak_moszynski2011}]
    Let $a,\ b,\ d \in \mathbb{R}$ satisfy $0<|b|<|d|$ and $|a|<|b+d|$. Then the solution $C_0$-semigroup
    to the Cauchy problem   (\ref{model_5}) is Devaney chaotic on  $\ell ^p$.
\end{theorem}

Actually, to show this result they used a spectral criterion (see \cite{banasiakmoszynski05} and \cite{deschschappacher97}) which is less general than the criterion of Corollary~\ref{cor:dsw}.
In consequence, we obtain that the solution $C_0$-semigroup
    to the Cauchy problem   (\ref{model_5}) admits an invariant mixing measure on $\ell ^p$ with full support.

\item
 In \cite{arozaperis}, Aroza and Peris studied the same model with variable coefficients,
 \begin{equation}\label{birthdeath}
    \begin{array}{ccclr}
      \frac{df_1}{dt} & =  & a_1f_1 + d_1f_2, &  \\
       \\
      \frac{df_n}{dt} & =  & b_nf_{n-1} + a_nf_n + d_nf_{n+1}, & \ \ n\geq 2.
    \end{array}
\end{equation}
  with $a_n,b_n,d_n\in\mathbb{R}$ and the infinite matrix
$$\mathcal{L}=\begin{pmatrix}
a_1&d_1&\empty & \empty & \empty\\
b_2&a_2&d_2&\empty &\empty\\
\empty &b_3&a_3&d_3 & \empty\\
\empty &\empty &b_4 & a_4 & \ddots\\
\empty & \empty & \empty & \ddots & \ddots
\end{pmatrix}.$$

They intended to obtain sub-chaos (i.e., Devaney chaos on a subspace) results for birth-and-death type models with proliferation in a wide range of variable coefficients. They considered the  Banach space $X$ on which the operator associated with $\mathcal{L}$ generates a $C_0$-semigroup. Given $1\leq p<\infty$, let
$$
X=X(\gamma):=\{f\in \ell^p:\mathcal{L}^nf\in\ell^p,\forall n\in\mathbb{N}, \mbox{ and } ||f||:=\sum_{n=0}^\infty||\mathcal{L}^nf||_p\gamma^{-n}<\infty\}.
$$
If the sequences $(a_n)_n,(b_n)_n$ and $(d_n)_n$ are bounded, $\mathcal{L}$ has an associated bounded operator $\mathcal{S}_p$ on $\ell^p$, with spectral radius $r(\mathcal(S_p)<\infty$, and $X(\gamma)=\ell^p$ for $\gamma>r(\mathcal{S}_p)$. If any of the sequences $(a_n)_n,(b_n)_n$ and $(d_n)_n$ is unbounded, we have that the operator $\mathcal{S}_X$ associated with $\mathcal{L}$ is a bounded operator on $X$ and, therefore, it generates a $C_0$-semigroup $\mathcal{T}_X$ on $X$. They obtained the following result:

\begin{theorem}[\cite{arozaperis}]
Let $(a_n),(b_n)$ and $(d_n)_n$ be sequences of real numbers such that $d_n\neq 0$ for all $n\in\mathbb{N}$, $1\leq p<\infty$, and $\gamma >0$. Assume that either
\begin{itemize}
\item [Case 1.] $\lim_{n\rightarrow \infty}a_n=a,\lim_{n\rightarrow \infty}b_n=b,\lim_{n\rightarrow \infty}d_n=d\neq 0$ with $|b|<|d|$ and $|a|<|b+d|$ or
\item [Case 2.] $\lim_{n\rightarrow \infty}\frac{a_n}{d_n}=\alpha,\lim_{n\rightarrow \infty}\frac{b_n}{d_n}=\beta,\lim_{n\rightarrow \infty}d_n=\infty$ with $\alpha^2\neq 4\beta$, $|\beta|<1$ and $|\alpha|<|1+\beta|$
\end{itemize}
then the $C_0$-semigroup $\mathcal{T}_X$ is sub-chaotic on $X(\gamma)$. Moreover, in case 1, $\mathcal{S}_p$ generates a sub-chaotic $C_0$-semigroup $\mathcal{T}_p$ on $\ell^p$.
\end{theorem}

Actually, to show this result they proved that the $C_0$-semigroup solution satisfies the spectral criterion of \cite{banasiakmoszynski05},
in particular the conditions of Corollary~\ref{cor:dsw} on a certain subspace $Y$. Thus, we obtain that the corresponding solution $C_0$-semigroup
   admits an invariant mixing measure $\mu$ on $X(\gamma)$ whose support is $Y$.

\item Let us consider the death model with variable coefficients
\begin{equation}\label{kinetikmodel}
    \left\{
        \begin{array}{ll}
            \frac{\partial f_n}{\partial t}=-\alpha_nf_n+\beta_nf_{n+1},& n\geq 1,\\\\
            f_n(0)=a_n, & n\geq 1\\\\
        \end{array}
    \right.
\end{equation}
 where $(\alpha_n)_n$ and $(\beta_n)_n$ are bounded positive sequences and $(a_n)_n \in \ell^1$ is a real sequence. Considering $X=\ell^1$, and the map $A$ given by
 $$
 Af=(-\alpha_nf_n+\beta_nf_{n+1})_n \mbox{ for } f=(f_n)_n\in X,
 $$
 since $A$ is a bounded operator on $\ell^1$, it generates a $C_0$-semigroup $(T_t)_{t\geq0}$ which is solution of \eqref{kinetikmodel}. It is shown in \cite{grosse-erdmann_peris2011linear}, that if
 $$
 \sup_{n\geq 1}\alpha_n<\liminf_{n\rightarrow\infty}\beta_n
 $$
 then the  semigroup $(T_t)_{t\geq0}$ satisfies the hypothesis of the spectral criterion \cite{deschschappacher97}, and then we we can ensure the existence of an invariant mixing measure with full support   on $X$.
  \end{enumerate}
\end{example}

\begin{example}
 Let us consider the  solution semigroup $(e^{tA})_{t\geq0}$ of the hyperbolic heat transfer equation problem:
\begin{equation}\label{heatproblem}
    \left\{
        \begin{array}{ll}
            \tau\frac{\partial^2u}{\partial t^2}+\frac{\partial u}{\partial t}=\alpha\frac{\partial^2u}{\partial x^2},\\\\
            u(0,x)=\varphi_1(x),   x\in\mathbb{R},\\\\
            \frac{\partial u}{\partial t}(0,x)=\varphi_2(x),  x\in\mathbb{R}
        \end{array}
    \right.
\end{equation}
 where $\varphi_1$ and $\varphi_2$ represent the initial temperature and the initial variation of temperature, respectively, $\alpha>0$ is the thermal diffusivity, and $\tau>0$ is the thermal relaxation time.
  We can represent it as a $C_0$-semigroup on the product of a certain function space with itself. We set $u_1=u$ and $u_2=\frac{\partial u}{\partial t}$. Then the associated first-order equation is :
\begin{equation}
\left\{
\begin{array}{c}
\frac{\partial}{\partial t}\left(
    \begin{array}{c}
      u_1\\
      u_{2} \\
    \end{array}
  \right)=\left(
            \begin{array}{cc}
              0 & I \\
              \frac{\alpha}{\tau}\frac{\partial^2}{\partial x^2} & \frac{-1}{\tau}I \\
            \end{array}
          \right)\left(
                   \begin{array}{c}
                     u_{1} \\
                     u_2 \\
                   \end{array}
                 \right).\\\\
  \left(
    \begin{array}{c}
      u_1(0,x) \\
      u_2(0,x) \\
    \end{array}
  \right)=  \left(
    \begin{array}{c}
      \varphi_1(x) \\
      \varphi_2(x)\\
    \end{array}
  \right)   ,  x\in\mathbb{R}
 \end{array}
    \right.
\end{equation}
 We fix $\rho>0$ and consider the space $$X_\rho=\{f:\mathbb{R}\rightarrow \mathbb{C}; f(x)=\sum_{n=0}^\infty\frac{a_n\rho^n}{n!}x^n, (a_n)_{n\geq0}\in c_0\}$$ endowed with the norm $||f||=\sup_{n\geq0}|a_n|$.

 Since

 \begin{equation}
A:=\left(
    \begin{array}{cc}
        0 & I \\
        \frac{\alpha}{\tau}\frac{\partial^2}{\partial x^2} & \frac{-1}{\tau}I \\
    \end{array}
    \right).
\end{equation}

is an operator on $X:=X_\rho\oplus X_\rho$, we have that $(e^{tA})_{t\geq0}$ is the $C_0$-semigroup solution of \ref{heatproblem}.
  We know from \cite{conejero_peris_trujillo} and \cite{grosse-erdmann_peris2011linear} that, given $\alpha$, $\tau$ and $\rho$ such that $\alpha\tau\rho>2$, then the solution semigroup $(e^{tA})_{t\geq0}$  defined on $X_\rho\oplus X_\rho$ satisfies the hypothesis of the spectral  criterion \cite{deschschappacher97}, and we conclude the existence of an invariant mixing measure with full support  on $X_\rho\oplus X_\rho$.
\end{example}

\begin{example}
In \cite{blackscholes73}, Black and Scholes proved that under some assumptions about the market, the value of a stock option $u(x,t)$, as a function of the current value of the underlying asset $x\in\mathbb{R}^+=[0,\infty)$ and time, satisfies the final value problem:
$$\left\{\begin{array}{ll}
\frac{\partial u}{\partial t}=-\frac{1}{2}\sigma^2x^2\frac{\partial^2u}{\partial x^2}-rx\frac{\partial u}{\partial x}+ru & \text{ in }\mathbb{R}^+\times[0,T]\\
u(0,T)=0& \text{ for } t\in[0,T]\\
u(x,T)=(x-p)^+&\text{ for } x \in\mathbb{R}^+
\end{array}\right.$$

where $p>0$ represents a given strike price , $\sigma >0$ is the volatility and $r>0$ is the interest rate. Let $v(x,t)=u(x,T-t)$. Then $v$ satisfies the forward Black-Scholes equation, defined for all time $t\in\mathbb{R}^+$ by
$$\left\{\begin{array}{ll}
\frac{\partial v}{\partial t}=\frac{1}{2}\sigma^2x^2\frac{\partial^2v}{\partial x^2}+rx\frac{\partial v}{\partial x}-rv & \text{ in }\mathbb{R}^+\times\mathbb{R}^{+}\\
v(0,T)=0& \text{ for } t\in\mathbb{R}^+\\
v(x,0)=f(x)&\text{ for } x \in\mathbb{R}^+
\end{array}\right.$$
 with
 $$f(x)=(x-p)^+=\left\{\begin{array}{ll}
x-p & \text{ if }x>p\\
0 & \text{ if }x\leq p.
\end{array}\right.$$

In order to  express this problem in an abstract form , we define $D_{\nu}=\nu x\frac{\partial}{\partial x}$, where $\nu=\frac{\sigma}{\sqrt{2}}$ and  $\mathcal{B}=(D_\nu)^2+\gamma(D_\nu)-rI$, with $\gamma=\frac{r}{\nu}-\nu$. Then the problem can be reformulated as:
$$\left\{\begin{array}{ll}
\frac{\partial v}{\partial t}=\mathcal{B}v, & \empty\\
v(0,T)=0, & \empty\\
v(x,0)=f(x)&\text{ for } x \in\mathbb{R}^+.
\end{array}\right.$$

 Recently \cite{goldstenmininniromanely}, gave a simple explicit representation of the solution of the Black-Scholes equation, and this representation holds in the spaces $Y^{s,\tau}$. Let
 $$
 Y^{s,\tau}=\{u\in C((0,\infty)) \ ; \ \lim_{x\rightarrow\infty}\frac{u(x)}{1+x^s}=0, \quad \lim_{x\rightarrow 0}\frac{u(x)}{1+x^{-\tau}}=0\}
 $$
 endowed with the norm
 $$
 ||u||_{Y^{s,\tau}}=\sup_{x>0}\biggl|\frac{u(x)}{(1+x^s)(1+x^{-\tau})}\biggr|.
 $$
 It is shown  that the $C_0$-semigroup solution of the Black-Scholes equation  can be represented by $T_t:=f(tD_\nu)$, where
 $$
 f(z)=e^{g(z)} \mbox{ with }  g(z)=z^2+\gamma z-r \mbox{ and }  D_{\nu}=\nu x\frac{\partial}{\partial x}.
 $$
For more information and details see \cite{blackscholes73}.

 In \cite{emiradgoldstein12},
  it is proved that the Black-Scholes semigroup is strongly continuous and chaotic for $s>1, \tau \geq 0$ with $s\nu>1$.
  We will see that, with a little more effort, the Black-Scholes semigroup satisfies the spectral criterion  in \cite{deschschappacher97}) under the same restrictions on the parameters and, therefore, the hypothesis of  Corollary~\ref{cor:dsw}.

Let $s>\frac{1}{\nu}$ , $0<\nu<1$ and $s>1$. Let $S_s=\{\lambda\in\mathbb{C} \ ; \ 0<Re\lambda<s\nu\}$. By Lemma 3.5 in \cite{emiradgoldstein12},  we have that $g(S_s)\cap i\mathbb{R}\neq\emptyset$. Then there exists an open ball $U\subset g(S_s)$ such that $U\cap i\mathbb{R}\neq \emptyset$ and such that $U\cap\mathbb{R}=\emptyset$. In particular, we find an inverse $g^{-1}$ well defined (and holomorphic) on $U$. We set $F=L\circ g^{-1}$, $F:U\rightarrow Y^{s,\tau}$, where $L:S_s\rightarrow Y^{s,\tau}$ is defined as $L(\lambda)=h_{\frac{\lambda}{\nu}}$, with $h_\lambda(x)=x^\lambda$. It is clear that $F$ is weakly holomorphic since $L$ is weakly holomorphic \cite{emiradgoldstein12}. Finally, $AF(\lambda )=g(\nu \frac{g^{-1}(\lambda)}{\nu})F(\lambda )=\lambda F(\lambda)$ for any $\lambda\in U$, where $(A,D(A))$ is the generator of the Black-Scholes semigroup, and the equality
$\langle F(\lambda ), \psi \rangle=0$ for a fixed $\psi\in (Y^{s,\tau})^*$ and for every $\lambda\in U$ necessarily implies $\psi=0$  \cite[Thm 3.6]{emiradgoldstein12}. Thus, the spectral criterion  in \cite{deschschappacher97} is satisfied and the Black-Scholes semigroup admits an invariant strong mixing Borel probability measure on $Y^{s,\tau}$ with full support by Corollary~\ref{cor:dsw}.
\end{example}

\end{document}